\documentclass[12pt]{amsart}
\usepackage{amssymb}
\usepackage{amsmath}
\usepackage{amsthm}
\usepackage{amsbsy}
\usepackage{psfrag}
\usepackage{graphics}
\usepackage{graphicx}
\theoremstyle{plain}
\newtheorem{theorem}{Theorem}[section]
\newtheorem{lemma}[theorem]{Lemma}

\theoremstyle{remark}
\newtheorem{remark}[theorem]{Remark}

\begin{document}
\allowdisplaybreaks[4]
\numberwithin{figure}{section}
\numberwithin{table}{section}
 \numberwithin{equation}{section}
%
\title[Convergent Adaptive DG Methods]
 {Convergence Analysis of the Lowest Order  Weakly Penalized Adaptive Discontinuous Galerkin Methods}




\author{Thirupathi Gudi}
\address{Department of Mathematics, Indian Institute of Science, Bangalore 56002}
\email{gudi@math.iisc.ernet.in}
\author{Johnny Guzm\'an}
\address{Division of Applied Mathematics,
 Brown University, Providence, RI 02912}
\email{Johnny\_Guzman@brown.edu}

\date{}
\begin{abstract}
In this article, we prove convergence of the weakly penalized adaptive discontinuous Galerkin methods. Unlike other works, we derive the
contraction property for various discontinuous Galerkin methods only assuming the stabilizing parameters are large enough to stabilize the method. A central
idea in the analysis is to construct an auxiliary solution from the discontinuous Galerkin solution by a simple post processing. Based on the auxiliary
solution, we define the adaptive algorithm  which guides to the convergence of adaptive discontinuous Galerkin methods.
\end{abstract}
\keywords{contraction, adaptive finite element, discontinuous Galerkin}
\subjclass{65N30, 65N15}
\maketitle
\def\d{\displaystyle}
\def\R{\mathbb{R}}
\def\cA{\mathcal{A}}
\def\cB{\mathcal{B}}
\def\cS{\mathcal{S}}
\def\p{\partial}
\def\O{\Omega}
\def\bbP{\mathbb{P}}
\def\bP{{\bf P}}
\def\bv{{\bf v}}
\def\bV{{\bf V}}
\def\bu{{\bf u}}
\def\bw{{\bf w}}
\def\bff{{\bf f}}
\def\cM{\mathcal{M}}
\def\cT{\mathcal{T}}
\def\cV{\mathcal{V}}
\def\cE{\mathcal{E}}
\def\mean#1{\left\{\hskip -5pt\left\{#1\right\}\hskip -5pt\right\}}
\def\jump#1{\left[\hskip -3.5pt\left[#1\right]\hskip -3.5pt\right]}
\def\smean#1{\{\hskip -3pt\{#1\}\hskip -3pt\}}
\def\sjump#1{[\hskip -1.5pt[#1]\hskip -1.5pt]}
\section{Introduction}\label{sec:Intro}
The design of adaptive finite element methods based on reliable
and efficient {\em a posteriori} error estimates  has been the
subject in the past
\cite{AO:2000:Book,BS:2001:Book,BR:2003:Book,BScott:2008:FEM,Verfurth:1995:AdaptiveBook}.
The adaptive finite element method  consists typically the
following successive loops of the sequence
\begin{center}
{\bf SOLVE $\rightarrow$ ESTIMATE $\rightarrow$ MARK $\rightarrow$ REFINE}
\end{center}
Convergence analysis of adaptive finite element
methods has been initiated by D\"orfler \cite{Dorfler:1996:AFEM}
who introduced an important marking strategy. Subsequently
important  theoretical developments have been made by many
researchers. We refer to
\cite{MNS:2000:AFEM,MNS:2002:AFEM,CKNS:2008:AFEM} for the work on
conforming finite element methods, to
\cite{CH:2006:AMFEM,CHX:2006:AMFEM} for the results on mixed
finite element methods, to \cite{CH:2006:ANCFEM,BMS:2010:ANFEM}
for the work on nonconforming methods and finally to
\cite{KP:2007:ADG,HKW:2008:ADG,BN:2010:ADG} for discontinuous
Galerkin methods. On the other hand, the optimality of adaptive
finite element method is derived in \cite{BDD:2004:Optimal} for
two dimensional problems and in \cite{Stevenson:2007:Optimality}
for high dimensional problems.

In this article, we focus on the low order adaptive discontinuous Galerkin
(DG) methods. Karakashian and Pascal \cite{KP:2007:ADG} were the first to prove
contraction properties for the symmetric interior penalty Galerkin
(SIPG) method. Therein, the authors have proved the
contraction property for SIPG method under an interior node
property. Subsequently the interior node property is relaxed
independently in the works of \cite{HKW:2008:ADG} and
\cite{BN:2010:ADG}. Moreover the quasi-optimal convergence rates
are derived in \cite{BN:2010:ADG}. However, the common issue with the three
articles \cite{KP:2007:ADG,HKW:2008:ADG,BN:2010:ADG} is that the contraction
property is derived assuming the penalty parameters are sufficiently large (i.e. larger than what is needed
for stability of the method).  In this article, we prove contraction properties for various symmetric weakly penalized
discontinuous Galerkin methods only assuming that the penalty parameters are large enough to guarantee stability of the method. For example, in the case of
the LDG method the stabilizing parameters only have to be positive.
This is achieved by a new marking strategy that uses an auxiliary solution obtained by post-processing the discontinuous
Galerkin solution which turns out to be the Crouzeix-Raviart non-conforming approximation \cite{CR:1973:NCFEM}. In fact, we borrow the marking strategies \cite{CH:2006:ANCFEM,BMS:2010:ANFEM} that have been developed for non-conforming methods and show that these are enough to contract the error of the entire DG approximation.

The weakly penalized method differ from classic DG methods in the fact that only the lower moments of the jumps are penalized on interfaces of the triangulation. For example, for piecewise linear elements and the SIPG method the penalty term looks like
\begin{equation*}
\sum_{e\in \cE_h}\frac{\alpha}{h_e} \int_e \sjump{w_h} \sjump{v_h}.
\end{equation*}
In contrast, in the weakly penalized case, one uses the term
\begin{equation*}
\sum_{e\in
\cE_h}\frac{\alpha}{h_e} \int_e \Pi_e\big(\sjump{w_h}\big)\Pi_e\big(\sjump{v_h}\big)
\end{equation*}
where we use the average of the jump $\Pi_e\big(\sjump{w_h}\big)=\frac{1}{h_e} \int_e \sjump{w_h}$. Note that this is equivalent to using the midpoint rule to evaluate the integrals $\int_e \sjump{w_h} \sjump{v_h}$, so the weakly penalized method is cheaper to implement. Weak penalization has been used in weakly over penalized methods \cite{BO:2007:WOPNIP, BOS:2008:WOPSIP}.

We consider the following model problem of finding $u\in
H^1_0(\Omega)$ such that
\begin{equation}\label{eq:MP}
a(u,v)=(f,v)\quad \forall v\in H^1_0(\O),
\end{equation}
where
\begin{equation}\label{eq:aForm}
a(w,v)=(\nabla w,\nabla v)\qquad \forall w,\,v\in H^1_0(\O),
\end{equation}
and $(\cdot,\cdot)$ denotes the $L^2(\O)$ inner product. We assume
that $\O\subset\R^2$ is a bounded domain with polygonal boundary
$\p\O$ and $f\in L^2(\Omega)$.

The rest of the article is organized as follows. In Section
\ref{sec:Notation}, we introduce the notation and preliminary
results. In Section \ref{sec:DGMethods}, we recall the DG methods
and corresponding stability results. In Section
\ref{sec:AuxiliarySolution}, we construct an auxiliary solution by
averaging the DG solution and there in we derive some useful
properties and results for the auxiliary solution. Section
\ref{sec:Convergence} is devoted to the convergence analysis of DG
methods. Finally we conclude the article in Section
\ref{sec:Conclusions}.

\section{Notation  and Preliminaries}\label{sec:Notation}
The following notation will be used throughout the article:
\begin{align*}
\cT_h &=\text{a face to face, shape regular simplicial triangulations of } \Omega\\
T&=\text{a triangle of } \cT_h \qquad
 h_T=\text{diameter of } T\\
\cE_h^i&=\text{set of all interior edges of } \cT_h\\
\cE_h^b&=\text{set of all boundary edges of } \cT_h\\
\cE_h&=\cE_h^i\cup\cE_h^b\\
\cM_T&=\text{set of midpoints of the edges of } T\\
\cM_h^i&=\text{set of all midpoints of the edges in } \cE_h^i\\
\cM_h^b&=\text{set of all midpoints of the edges in } \cE_h^b\\
\cM_h&=\cM_h^i\cup\cM_h^b\\
h_e&= \text{ length of the edge } e\in\cE_h\\
\nabla_h &=\text{piecewise (element-wise) gradient}\\
\bbP_m(T)&=\text{space of polynomials of degree less than or equal to } m\geq 0 \text{ and defined on }T.
\end{align*}
The discontinuous finite element space is defined by
$$V_h=\{v_h\in L^2(\O): v_h|_T\in P_1(T)\}.$$
In the analysis below, we need the following Crouzeix-Raviart
nonconforming space \cite{CR:1973:NCFEM}:
$$V_{CR}=\{v_h\in V_h: \int_e \sjump{v_h}\,ds =0 \quad \forall e\in \cE_h\},$$
and the following vector valued discrete space:
$$W_h=\{w_h\in L^2(\O)^2: w_h|_T\in [P_0(T)]^2\}.$$
\par
Define a broken Sobolev space
\begin{eqnarray*}
H^1(\O,\cT_h)=\{v\in L_2(\Omega) :\,v_T= v|_{T}\in
H^1(T)\quad\forall\,~T\in{\mathcal T}_h\}.
\end{eqnarray*}
\par
For the DG methods, we require to define  jump and mean of
discontinuous functions.
 For any $e\in\cE_h^i$, there are two triangles $T_+$ and $T_-$ such that
 $e=\partial T_+\cap\partial T_-$.
 Let $n_-$ be the unit normal of $e$ pointing from $T_-$ to $T_+$, and $n_+=-n_-$.
 (cf. Fig.~\ref{Fig1}).
 For any $v\in H^1(\Omega,{\mathcal T}_h)$, we define the jump and mean of $v$ on $e$ by
\begin{eqnarray*}
 \sjump{v} = v_-n_-+v_+n_+,\,\,\mbox{and}\,\,\smean{v} = \frac{1}{2}(v_-+v_+),\,\mbox{respectively,}
\end{eqnarray*}
 where $v_\pm=v\big|_{T_\pm}$.
Similarly define for $w\in H^1(\Omega,{\mathcal T}_h)^2$ the jump
and mean of $w$ on $e\in\cE_h^i$ by
\begin{eqnarray*}
 \sjump{w} = w_-\cdot n_-+w_+\cdot n_+,\,\,\mbox{and}\,\,\smean{w} = \frac{1}{2}(w_-+w_+),\,\mbox{respectively,}
\end{eqnarray*}
 where $w_\pm=w |_{T_\pm}$.
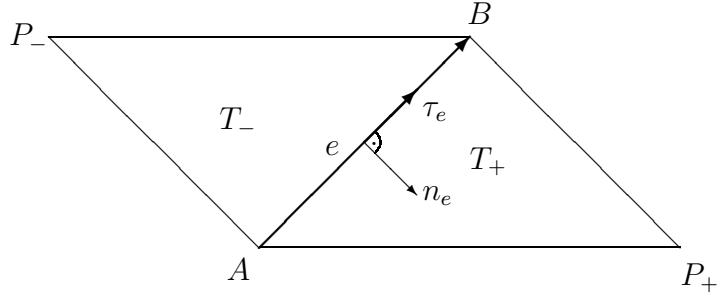
\begin{figure}[!hh]
\begin{center}
\setlength{\unitlength}{0.7cm}
\begin{picture}(8,6)
\put(2,1){\line(1,0){8}} \put(2,1){\line(-1,1){4}}
\put(6,5){\line(-1,0){8}} \put(10,1){\line(-1,1){4}}
\thicklines\put(2,1){\vector(1,1){4}}\thinlines
\put(1.39,0.39){$A$} \put(5.95,5.25){$B$} \put(10,0.29){$P_{+}$}
\put(-2.75,4.85){$P_{-}$} \put(1.25,3.25){$T_{-}$}
\put(6,2.5){$T_{+}$} \put(4,3){\vector(1,-1){1}}
\put(5.10,1.90){$n_{e}$}
\thicklines\put(4,3){\vector(1,1){1}}\thinlines
\put(5.10,3.50){$\tau_{e}$} \put(3.25,2.75){$e$}
\qbezier(4.2,2.8)(4.4,3.0)(4.2,3.2) \put(4.1,2.95){.}
\end{picture}
\caption{Two neighboring triangles $T_-$ and $T_{+}$ that share
the edge $e=\partial T_-\cap\partial T_+$ with initial node $A$
and end node $B$ and unit normal $n_e$. The orientation of $n_e =
n_{-} = -n_{+}$ equals the outer normal of $T_{-}$, and hence,
points into $T_{+}$.} \label{Fig1}
\end{center}
\end{figure}
\par
 For any edge $e\in \cE_h^b$,  there is a triangle $T\in\cT_h$ such that
 $e=\partial T\cap \partial\Omega$.
 Let $n_e$ be the unit normal of $e$ that points outside $T$.
 For any $v\in H^1(T)$, we set on $e\in \cE_h^b$
\begin{eqnarray*}
 \sjump{v}= vn_e\,\,\mbox{and}\,\,\smean{v}=v,
\end{eqnarray*}
and for $w\in H^1(T)^2$,
\begin{eqnarray*}
\sjump{w}= w\cdot n_e,\,\,\mbox{and}\,\,\smean{w}=w.
\end{eqnarray*}

The discontinuous Galerkin methods use a lifting operator
$r:L^2(\cE_h)^2 \rightarrow W_h $ defined by
\begin{align}\label{eq:Lifting}
\int_\Omega r(w)\cdot \tau \,dx = -\sum_{e\in\cE_h} \int_e w\cdot
\smean {\tau}\,ds \quad \forall \tau \in W_h,
\end{align}
and a local analogue $r_e:L^2(e)^2 \rightarrow W_h $ defined by
\begin{align}\label{eq:LiftingLocal}
\int_\Omega r_e(w)\cdot \tau \,dx = -\int_e w\cdot \smean {\tau}\,ds \quad \forall \tau \in W_h.
\end{align}
Let $\Pi_e:L^2(e)\rightarrow R$ be the $L^2$-projection onto
constants defined by
\begin{equation}\label{eq:Def-Pi0}
\Pi_e(v)=\frac{1}{h_e}\int_e v\,ds.
\end{equation}
\section{Discontinuous Galerkin Methods}\label{sec:DGMethods}
We consider four stable and symmetric weakly penalized
discontinuous Galerkin Methods namely, the IP (or SIPG method)
method \cite{DD:1976:IP,Wheeler::1978:IP,Arnold:1982:IPD}, the LDG
method \cite{CS:1998:LDG,ABCM:2002:UnifiedDG}, the method by
Brezzi et al. \cite{BMMPR:2000:DG} and the method by Bassi et al.
\cite{BRMPS:1997:DG}. The original articles have considered the
formulations using the penalty term with strong jumps, we replace
them here with weak jumps. The concept of stabilizing the DG
formulation by weak jumps was introduced in
\cite{BO:2007:WOPNIP,BOS:2008:WOPSIP}.

\par
\noindent The bilinear form for the IP method
\cite{DD:1976:IP,Wheeler::1978:IP,Arnold:1982:IPD} is defined by
\begin{align}\label{eq:FormIP}
\cA_h(w_h,v_h)&=\sum_{T\in \cT_h} \int_T \nabla w_h\cdot \nabla v_h\,dx -\sum_{e\in \cE_h}\int_e \smean{\nabla w_h}\sjump{v_h}\,ds\\
&-\sum_{e\in \cE_h}\int_e \smean{\nabla
v_h}\sjump{w_h}\,ds+\sum_{e\in
\cE_h}\frac{\alpha}{h_e} \int_e \Pi_e\big(\sjump{w_h}\big)\Pi_e\big(\sjump{v_h}\big).\notag
\end{align}
where $\alpha>0$ is the stabilizing parameter.
\par
\noindent The bilinear form for the LDG method
\cite{CS:1998:LDG,ABCM:2002:UnifiedDG} is defined  by
\begin{align}\label{eq:FormLDG}
\cA_h(w_h,v_h)&=\sum_{T\in \cT_h} \int_T \nabla w_h\cdot \nabla
v_h\,dx -\sum_{e\in \cE_h}\int_e \smean{\nabla w_h}\sjump{v_h}\,ds
-\sum_{e\in \cE_h}\int_e \smean{\nabla v_h}\sjump{w_h}\,ds
\\&\quad+\int_\O
r(\Pi_e(\sjump{w_h}))r(\Pi_e(\sjump{v_h}))+\sum_{e\in
\cE_h}\frac{\alpha}{h_e} \int_e \Pi_e\big(\sjump{w_h}\big)\Pi_e\big(\sjump{v_h}\big).\notag
\end{align}
The bilinear form for the Brezzi et al. method
\cite{BMMPR:2000:DG} is defined by
\begin{align}\label{eq:FormBrezzi}
\cA_h(w_h,v_h)&=\sum_{T\in \cT_h} \int_T \nabla w_h\cdot \nabla
v_h\,dx -\sum_{e\in \cE_h}\int_e \smean{\nabla w_h}\sjump{v_h}\,ds
-\sum_{e\in \cE_h}\int_e \smean{\nabla
v_h}\sjump{w_h}\,ds\\&\quad+\int_\O
r(\Pi_e(\sjump{w_h}))r(\Pi_e(\sjump{v_h}))+\sum_{e\in
\cE_h}\alpha \int_\Omega
r_e\big(\Pi_e(\sjump{w_h})\big)r_e\big(\Pi_e(\sjump{v_h})\big).\notag
\end{align}
The bilinear form for the Bassi et al. method \cite{BRMPS:1997:DG}
is defined by
\begin{align}\label{eq:FormBassi}
\cA_h(w_h,v_h)&=\sum_{T\in \cT_h} \int_T \nabla w_h\cdot \nabla
v_h\,dx -\sum_{e\in \cE_h}\int_e \smean{\nabla w_h}\sjump{v_h}\,ds
-\sum_{e\in \cE_h}\int_e \smean{\nabla v_h}\sjump{w_h}\,ds\\&\quad
+\sum_{e\in \cE_h}\alpha\int_\Omega
r_e\big(\Pi_e(\sjump{w_h})\big)r_e\big(\Pi_e(\sjump{v_h})\big).\notag
\end{align}
\par
\noindent The DG method is to find $u_h\in V_h$ such that
\begin{align}\label{eq:DGM}
\cA_h(u_h,v_h)=(f,v_h)\qquad \forall v_h\in V_h,
\end{align}
where $\cA_h$ is any of the bilinear form defined in
\eqref{eq:FormIP}-\eqref{eq:FormBassi}.
\par
\noindent It is proved in \cite[Lemma 1]{Ainsworth:2007:ApostDG}
that the IP method is stable for any $\alpha$ satisfying
\begin{align}\label{eq:CondSIPG}
\alpha> 4 \max_{T\in\cT_h} \rho(S_T),
\end{align}
where $\rho(S_T)$ is the spectral radius of the local stiffness
matrix $[S_T]_{mn}=(\nabla_h\lambda_m,\nabla_h\lambda_n)$, where
$\lambda_m$'s are barycentric coordinates of $T$. In the Table
\ref{table:alpha}, we present the condition on $\alpha$ for the
above DG methods.
\begin{table}[h!!]
 {\small{\footnotesize
\begin{center}
\begin{tabular}{|c|c|}\hline
   Method & Condition on $\alpha $   \\
\hline\\[-12pt]  &\\[-8pt]
IP Method \cite{DD:1976:IP,Wheeler::1978:IP,Arnold:1982:IPD} & $\alpha$ satisfies \eqref{eq:CondSIPG}   \\
LDG Method  \cite{CS:1998:LDG,ABCM:2002:UnifiedDG} & $\alpha>0$  \\
Brezzi et al. \cite{BMMPR:2000:DG} & $\alpha>0$    \\
Bassi et al.  \cite{BRMPS:1997:DG} & $\alpha>3$   \\
\hline
\end{tabular}
\end{center}
}}
\par\medskip
\caption{Conditions for $\alpha$} \label{table:alpha}
\end{table}

\par
\noindent Define the mesh dependent norm
\begin{align}\label{eq:hnorm}
\|v\|_{1,h}^2=\sum_{T\in\cT_h}\int_T |\nabla v|^2\,dx
+\sum_{e\in\cE_h}\Pi_e(\sjump{v})^2 \quad\forall v\in
H^1(\O,\cT_h).
\end{align}
The following lemma on the stability of the DG methods
\eqref{eq:DGM} is well-known \cite{ABCM:2002:UnifiedDG}.
\begin{lemma}\label{lem:Coercivity}
Assume that $\alpha$ satisfies the conditions in Table
$\ref{table:alpha}$. Then, it holds that
\begin{align*}
C\|v_h\|_{1,h}^2 \leq \cA_h(v_h,v_h)\qquad \forall v_h\in V_h.
\end{align*}
\end{lemma}
\section{An Auxiliary Solution by Post
Processing}\label{sec:AuxiliarySolution}

Let $u_h\in V_h$ be the solution of any of the DG methods
\eqref{eq:DGM}. Define an auxiliary solution $u_h^*\in V_{CR}$ by
the following:
\begin{equation}\label{eq:Def-uh*}
u_h^*(m_e) := \left\{ \begin{array}{ll} \smean{u_h}(m_e), & \text{if } m_e\in \cM_h^i\\\\
0, & \text{if } m_e\in \cM_h^b.
\end{array}\right.
\end{equation}
In the following lemma, we establish an integral relation for
$u_h$ and $u_h^*$.
\begin{lemma}\label{lem:Relations-uh-uh*}
For any $v_h\in V_h$, it holds that
\begin{align}\label{eq:Relations-uh-uh*}
\sum_{T\in \cT_h} \int_T \nabla u_h\cdot \nabla v_h\,dx -\sum_{e\in \cE_h}\int_e \smean{\nabla v_h}\sjump{u_h}\,ds=\sum_{T\in \cT_h} \int_T \nabla
u_h^*\cdot \nabla v_h\,ds.
\end{align}
\end{lemma}
\begin{proof}
Using integration by parts, we find
\begin{align*}
\sum_{T\in \cT_h} \int_T \nabla (u_h-u_h^*)\cdot \nabla v_h\,dx=\sum_{e\in \cE_h}\int_e \smean{\nabla v_h}\sjump{u_h-u_h^*}\,ds+\sum_{e\in \cE_h^i}\int_e
\sjump{\nabla v_h}\smean{u_h-u_h^*}\,ds.
\end{align*}
The definition of $u_h^*$ implies
\begin{align*}
\sum_{e\in \cE_h^i}\int_e \sjump{\nabla v_h}\smean{u_h-u_h^*}\,ds&=0,\\
\sum_{e\in \cE_h}\int_e \smean{\nabla v_h}\sjump{u_h^*}\,ds&=0.
\end{align*}
This completes the proof.
\end{proof}
In the following lemma, we estimate the error between $u_h$ and
$u_h^*$.
\begin{lemma}\label{lem:Error-uh-uh*}
It holds that
\begin{align*}
\sum_{T\in \cT_h}\left( h_T^{-2} \|u_h-u_h^*\|_{L^2(T)}^2 +
\|\nabla(u_h-u_h^*)\|_{L^2(T)}^2 \right)\leq C \sum_{e\in
\cE_h}\Pi_e\big(\sjump{u_h}\big)^2
\end{align*}
\end{lemma}
\begin{proof}
The proof is an easy consequence of the following estimate: for
any $v_h\in P_1(T)$, it holds that
\begin{align}
\|v_h\|_{L^2(T)}^2 \leq C |T|\sum_{m_e\in\cM_T} v_h(m_e)^2.
\end{align}
\end{proof}
The following identity is useful in our subsequent analysis.
\begin{lemma}\label{lem:MethodsRelation}
It holds that
\begin{align*}
\cA_h(u_h-u_h^*,u_h-u_h^*)=(f,u_h-u_h^*)
\end{align*}
\end{lemma}
\begin{proof}
Using \eqref{eq:DGM}, we find
\begin{align*}
\cA_h(u_h-u_h^*,u_h-u_h^*)-(f,u_h-u_h^*)=-\cA_h(u_h^*,u_h-u_h^*)
\end{align*}
It is remaining to show that $\cA_h(u_h^*,u_h-u_h^*)=0$. Using the
fact that  $\Pi_e(\sjump{u_h^*})=0$ for all $e\in \cE_h$,
integration by parts and \eqref{eq:Def-uh*}, we obtain
\begin{align*}
\cA_h(u_h^*,u_h-u_h^*)&=\sum_{T\in\cT_h}\int_T \nabla u_h^*\cdot \nabla(u_h-u_h^*)dx-\sum_{e\in\cE_h}\int_e \smean{\nabla u_h^*}\sjump{u_h-u_h^*}ds\\
&=\sum_{e\in\cE_h^i}\int_e \sjump{\nabla u_h^*}\smean{u_h-u_h^*}ds\\
&=0.
\end{align*}
This completes the proof.
\end{proof}
\par
\noindent In the following lemma, we estimate the error
$u_h^*-u_h$ by volume residual.
\begin{lemma}\label{lem:JumpsbyVolume}
Assume that $\alpha$ satisfies the conditions in Table $\ref{table:alpha}$. Then there exists some $C^*>0$ such  that
\begin{align*}
\|u_h^*-u_h\|_{1,h}^2 \leq C^* \sum_{T\in\cT_h}h_T^2 \|f\|_{L^2(T)}^2.
\end{align*}
\end{lemma}
\begin{proof}
Using Lemma \ref{lem:Coercivity}, Lemma \ref{lem:MethodsRelation},
Cauchy-Schwarz inequality, Lemma \ref{lem:Error-uh-uh*} and
Young's inequality, we find
\begin{align*}
C\|u_h^*-u_h\|_{1,h} &\leq \cA_h(u_h^*-u_h,u_h^*-u_h)=(f,u_h^*-u_h)\\
&\leq \left(\sum_{T\in\cT_h}h_T^2 \|f\|_{L^2(T)}^2\right)^{1/2}\left(\sum_{T\in\cT_h}h_T^{-2} \|u_h^*-u_h\|_{L^2(T)}^2\right)^{1/2}\\
&\leq C \left(\sum_{T\in\cT_h}h_T^2 \|f\|_{L^2(T)}^2\right)^{1/2}\left(\sum_{e\in\cE_h}\Pi_e(\sjump{u_h})^2\right)^{1/2}\\
&\leq \frac{C}{\epsilon}\left(\sum_{T\in\cT_h}h_T^2 \|f\|_{L^2(T)}^2\right)+\epsilon \|u_h^*-u_h\|_{1,h}^2.
\end{align*}
We complete the proof by choosing $\epsilon$ sufficiently small.
\end{proof}
\begin{lemma}\label{lem:CRSolution}
The auxiliary solution $u_h^*$ satisfies
\begin{equation*}
(\nabla_hu_h^*,\nabla_h v_h)=(f,v_h)\quad \forall v_h\,\in V_{CR},
\end{equation*}
i.e, $u_h^*$ is the solution of the classical nonconforming method
\cite{CR:1973:NCFEM}.
\end{lemma}
\begin{proof}
Using \eqref{eq:DGM} for any $v_h\in V_{CR}$, we find
\begin{align*}
\sum_{T\in \cT_h} \int_T \nabla u_h\cdot \nabla v_h\,dx
-\sum_{e\in \cE_h}\int_e \smean{\nabla
v_h}\sjump{u_h}\,ds=(f,v_h).
\end{align*}
Then using \eqref{eq:Relations-uh-uh*}, we complete the proof.
\end{proof}
The following {\em a posteriori} error estimator is an easy
consequence of Lemma \ref{lem:JumpsbyVolume} and the results in
\cite{CH:2006:ANCFEM}:
\begin{lemma}
Let $u$ and $u_h$ be the solutions of $\eqref{eq:MP}$ and
$\eqref{eq:DGM}$. Let $u_h^*$ be the auxiliary solution defined in
$\eqref{eq:Def-uh*}$. Then it holds that
\begin{align*}
\|\nabla_h(u-u_h)\|_{L^2(\Omega)} \leq C
\left(\sum_{T\in\cT_h}h_T^2 \|f\|_{L^2(T)}^2+\sum_{e\in
\cE_h^i}\int_e h_e\sjump{\partial u_h^*/\partial s}^2\,ds
\right)^{1/2},
\end{align*}
where $\partial/\partial s$ denotes the tangential derivative
along the edge $e$.
\end{lemma}
\begin{proof}
First using triangle inequality
\begin{equation*}
\|\nabla_h(u-u_h)\|_{L^2(\Omega)} \leq
\|\nabla_h(u-u_h^*)\|_{L^2(\Omega)}
+\|\nabla_h(u_h^*-u_h)\|_{L^2(\Omega)},
\end{equation*}
and then using Lemma \ref{lem:JumpsbyVolume} and the results in
\cite{CH:2006:ANCFEM}, we complete the proof.
\end{proof}
\section{Convergence of Adaptive DG
Methods}\label{sec:Convergence}
Let $u_h\in V_h$ be the solution of any of the DG method
\eqref{eq:DGM} and let $u_h^*\in V_{CR}$ be the auxiliary solution
constructed in \eqref{eq:Def-uh*}. Using Lemma
\ref{lem:CRSolution}, recall that $u_h^*$ is the solution of the
Crouzeix-Raviart nonconforming method and Lemma
\ref{lem:JumpsbyVolume} implies that there exists a positive
constant $C^*>0$ such that
\begin{align*}
\|u_h^*-u_h\|_h^2+\sum_{e\in\cE_h}\Pi_e(\sjump{u_h})^2 \leq C^*
\|hf\|^2,
\end{align*}
where hereafter $\|\cdot\|_h=\|\nabla_h\cdot\|$ and
$$\|hf\|^2=\sum_{T\in\cT_h}h_T^2 \|f\|_{L^2(T)}^2.$$
\par
\noindent Let  $\cT_h$ be the conforming refinement of $\cT_H$
obtained by refining the all the marked elements in $\cT_H$ that
are marked in the step {\bf MARK}. The functions with $h$ (resp.
$H$) suffix corresponds to the mesh $\cT_h$ (resp. $\cT_H$).
Below, we consider separately two different marking strategies
that are introduced by Carstensen and Hoppe \cite{CH:2006:ANCFEM}
and by Becker, Mao and Shi \cite{BMS:2010:ANFEM} and prove the
error reduction for both the algorithms separately.
\par\smallskip
\noindent {\em Marking by Carstensen and Hoppe
\cite{CH:2006:ANCFEM}:

Given the universal constants  with $0 <\Theta, \rho_2 < 1$, the
outcome of MARK is a set of edges $\cM \subset\cE_H$ such that
\begin{equation}\label{eq:MarkingCH}
\Theta \sum_{e\in \cE_H^i}\int_e H_e\sjump{\partial u_H^*/\partial
s}^2\,ds \leq \sum_{e\in \cM}\int_e H_e\sjump{\partial
u_H^*/\partial s}^2\,ds.
\end{equation}
The refined regular triangulation $\cT_h$ from REFINE generated by
refining at least all the edges in $\cM$ (and possibly further
edges to avoid hanging nodes) with the new mesh-size $h< H$ is
supposed to satisfy
\begin{equation*}
\rho_2\|Hf\|_{L^2(\Omega)}^2\leq \|hf\|_{L^2(\Omega)}^2.
\end{equation*}}
\par
The results in Carstensen  and Hoppe \cite{CH:2006:ANCFEM} imply
that there exists $0<\rho_1<1$  and $C_1>0$ such that
\begin{align}
\|u-u_h^*\|_h^2 &\leq \rho_1 \|u-u_H^*\|_H^2 +C_1 \|Hf\|^2,\label{eq:NC1}\\
\|hf\|^2 &\leq \rho_2 \|Hf\|^2. \label{eq:NC2}
\end{align}
\par
\noindent In the following theorem, we derive the contraction
property for the adaptive DG methods \eqref{eq:DGM} using the
marking strategy by Carstensen  and Hoppe \eqref{eq:MarkingCH}.
\begin{theorem}\label{thm:ContractionCH}
Let the marking be done by $\eqref{eq:MarkingCH}$. Then there
exists $\gamma>0$ and $0<\rho^* <1$ such that
$$\|u-u_h\|_h^2+\gamma \|hf\|^2 \leq \rho^*\, \left(\|u-u_H\|_H^2+\gamma \|Hf\|^2\right).$$
\end{theorem}
\begin{proof}
Let $\epsilon>0$. Using triangle inequality and Young's
inequality, we find
\begin{align*}
\|u-u_h\|_h^2+\gamma\|hf\|^2 \leq (1+\epsilon) \|u-u_h^*\|_h^2+ (1+1/\epsilon)\|u_h-u_h^*\|_h^2+\gamma \|hf\|^2.
\end{align*}
Using Lemma \ref{lem:JumpsbyVolume}, we obtain
\begin{align*}
\|u-u_h\|_h^2+\gamma\|hf\|^2 \leq (1+\epsilon) \|u-u_h^*\|_h^2+ \big(C^*(1+1/\epsilon)+\gamma\big)\|hf\|^2.
\end{align*}
Using the error reduction for $u_h^*$ \eqref{eq:NC1}, we find
\begin{align*}
\|u-u_h\|_h^2+\gamma\|hf\|^2 \leq (1+\epsilon)\rho_1 \big( \|u-u_H^*\|_H^2+C_1 \|Hf\|^2\big)+ \big(C^*(1+1/\epsilon)+\gamma\big)\|hf\|^2.
\end{align*}
Again using triangle inequality, Young's inequality and lemma
\ref{lem:JumpsbyVolume}, we find
\begin{align*}
\|u-u_h\|_h^2+\gamma\|hf\|^2 &\leq (1+\epsilon)\rho_1 \big( (1+\epsilon)\|u-u_H\|_H^2+C^*(1+1/\epsilon)\|Hf\|^2+C_1 \|Hf\|^2\big)\\
& \quad +\big(C^*(1+1/\epsilon)+\gamma\big)\|hf\|^2.
\end{align*}
Therefore
\begin{align*}
\|u-u_h\|_h^2+\gamma\|hf\|^2 &\leq (1+\epsilon)^2\rho_1 \|u-u_H\|_H^2+(1+\epsilon)\rho_1 \big(C^*(1+1/\epsilon)+C_1\big)\|Hf\|^2\\
&+ \big(C^*(1+1/\epsilon)+\gamma\big)\|hf\|^2.
\end{align*}
We now use \eqref{eq:NC2} and find
\begin{align*}
\|u-u_h\|_h^2+\gamma\|hf\|^2 &\leq (1+\epsilon)^2\rho_1 \|u-u_H\|_H^2+(1+\epsilon)\rho_1 \big(C^*(1+1/\epsilon)+C_1\big)\|Hf\|^2\\
& \quad +\big(C^*(1+1/\epsilon)+\gamma\big)\rho_2 \|Hf\|^2.
\end{align*}
By simplifying
\begin{align*}
\|u-u_h\|_h^2+\gamma\|hf\|^2 &\leq (1+\epsilon)^2\rho_1
\|u-u_H\|_H^2\\&\quad +\left[(1+\epsilon)\rho_1
\big(C^*(1+1/\epsilon)+C_1\big)+
\big(C^*(1+1/\epsilon)\big)\rho_2+\gamma\rho_2 \right]\|Hf\|^2.
\end{align*}
Choosing $\gamma$ sufficiently large such that
\begin{equation*}
\left[(1+\epsilon)\rho_1 \big(C^*(1+1/\epsilon)+C_1\big)+ \big(C^*(1+1/\epsilon)\big)\rho_2+\gamma\rho_2 \right] \leq \gamma (1+\rho_2)/2.
\end{equation*}
That is by choosing $\gamma$ such that
\begin{equation*}
2\left[(1+\epsilon)\rho_1 \big(C^*(1+1/\epsilon)+C_1\big)+ \big(C^*(1+1/\epsilon)\big)\rho_2\right]/(1-\rho_2) \leq \gamma,
\end{equation*}
we complete the proof by $\rho^*=\min\{(1+\epsilon)^2\rho_1,(1+\rho_2)/2\}$ and $\epsilon$ such that $(1+\epsilon)^2\rho_1<1$.
\end{proof}
\par
\noindent {\em Marking by Becker, Mao and Shi
\cite{BMS:2010:ANFEM}:  Choose the parameters $0<\theta,\sigma<1$
and $\gamma$. Then the out come of MARK step is the set of edges or elements according to the following: \\
If $\|Hf\|_{L^2(\Omega)}^2 \leq \gamma \sum_{e\in \cE_H^i}\int_e
H_e\sjump{\partial u_H^*/\partial s}^2\,ds$, then mark a subset
$\cM\subset \cE_h^i$ with minimal cardinality such that
\begin{equation}\label{eq:MarkingBMSj}
\theta \sum_{e\in \cE_H^i}\int_e H_e\sjump{\partial u_H^*/\partial
s}^2\,ds \leq \sum_{e\in \cM}\int_e H_e\sjump{\partial
u_H^*/\partial s}^2\,ds,
\end{equation}
else find a subset $\cM_1\subset \cT_H$ with minimal cardinality
such that
\begin{equation}\label{eq:MarkingBMSv}
\sigma \sum_{T\in \cT_H}H_T^2 \|f\|_{L^2(T)}^2 \leq \sum_{T\in
\cM_1} H_T^2 \|f\|_{L^2(T)}^2.
\end{equation}
}


\par
The result \cite{BMS:2010:ANFEM} by Becker, Mao and Shi is that
there exist $0<\rho <1$  and $\beta^*>0$ such that
\begin{align}
\|u-u_h^*\|_h^2+\beta\|hf\|^2 &\leq \rho \big( \|u-u_H^*\|_H^2
+\beta \|Hf\|^2 \big),\label{eq:NC3}
\end{align}
for all $\beta$ such that $\beta\geq \beta^*$.
\par
\noindent Below, we prove the convergence of adaptive DG methods
under the Becker, Mao and Shi marking
\eqref{eq:MarkingBMSj}-\eqref{eq:MarkingBMSv}.

\begin{theorem}\label{thm:ContractionBMS} Suppose that the marking
is done by $\eqref{eq:MarkingBMSj}$-$\eqref{eq:MarkingBMSv}$. Then
there exists $\gamma>0$ and $0<\rho^* <1$ such that
$$\|u-u_h\|_h^2+\gamma \|hf\|^2 \leq \rho^*\, \left(
\|u-u_H\|_H^2+\gamma \|Hf\|^2\right).$$
\end{theorem}
\begin{proof}
Let $\epsilon>0$. Using triangle inequality and Young's
inequality, we find
\begin{align*}
\|u-u_h\|_h^2+\gamma\|hf\|^2 \leq (1+\epsilon) \|u-u_h^*\|_h^2+ (1+1/\epsilon)\|u_h-u_h^*\|_h^2+\gamma \|hf\|^2.
\end{align*}
Using Lemma \ref{lem:JumpsbyVolume}, we obtain
\begin{align*}
\|u-u_h\|_h^2+\gamma\|hf\|^2 \leq (1+\epsilon) \|u-u_h^*\|_h^2+
\big(C^*(1+1/\epsilon)+\gamma\big)\|hf\|^2,
\end{align*}
or equivalently
\begin{align*}
\|u-u_h\|_h^2+\gamma\|hf\|^2 \leq (1+\epsilon) \Big(\|u-u_h^*\|_h^2+ \big(C^*(1+1/\epsilon)+\gamma\big)/(1+\epsilon)\|hf\|^2\Big).
\end{align*}
Assume that $\gamma$ is sufficiently large such that
\begin{equation}\label{eq:Cond1}
\big(C^*(1+1/\epsilon)+\gamma\big)/(1+\epsilon) =: \beta\geq
\beta^*.
\end{equation}
Then, using \eqref{eq:NC3} we arrive at
\begin{align*}
\|u-u_h\|_h^2+\gamma\|hf\|^2 \leq (1+\epsilon)\rho \big(
\|u-u_H^*\|_H^2+\beta \|Hf\|^2\big).
\end{align*}
Again using triangle inequality and Young's inequality and Lemma
\ref{lem:JumpsbyVolume}, we find
\begin{align*}
\|u-u_h\|_h^2+\gamma\|hf\|^2 &\leq (1+\epsilon)\rho \left[
(1+\epsilon)\|u-u_H\|_H^2+C^*(1+1/\epsilon)\|Hf\|^2+\beta
\|Hf\|^2\right].
\end{align*}
Therefore
\begin{align*}
\|u-u_h\|_h^2+\gamma\|hf\|^2 &\leq (1+\epsilon)^2\rho
\|u-u_H\|_H^2+(1+\epsilon)\rho \big(C^*(1+1/\epsilon)+\beta
\big)\|Hf\|^2.
\end{align*}
First note that we can choose $\epsilon$ sufficiently small such
that $(1+\epsilon)^2\rho<1$. Then the proof will be completed if
we can show that there is some $0<\rho_2<1$ such that
\begin{equation}\label{eq:Cond2}
(1+\epsilon)\rho \big(C^*(1+1/\epsilon)+\beta \big) \leq \rho_2
\gamma,
\end{equation}
with $\rho^*=\min\{(1+\epsilon)^2\rho, \rho_2\}$.
\par
\noindent Using \eqref{eq:Cond1} in \eqref{eq:Cond2},
\begin{equation*}
(1+\epsilon)\rho C^*(1+1/\epsilon)+\frac{(1+\epsilon)\rho
C^*(1+1/\epsilon)}{(1+\epsilon)}+ \rho\gamma \leq \rho_2 \gamma,
\end{equation*}
equivalently
\begin{equation*}\label{eq:Cond21}
(1+\epsilon)\rho C^*(1+1/\epsilon)+\frac{(1+\epsilon)\rho
C^*(1+1/\epsilon)}{(1+\epsilon)} \leq (\rho_2-\rho) \gamma.
\end{equation*}
The proof is completed by choosing $\rho_2$ such that
$\rho<\rho_2<1$ and $\gamma$ sufficiently large.
\end{proof}

\begin{remark}
Using triangle inequality and Lemma \ref{lem:JumpsbyVolume}, we
find
\begin{align*}
\|u-u_h\|_h &\leq \|u-u_h^*\|+\|u_h^*-u_h\|_h\\ &\leq
\|u-u_h^*\|+C \|hf\|.
\end{align*}
Therefore the adaptive DG methods converge at least at the rate of
adaptive nonconforming method.

\end{remark}

\section{Conclusions}\label{sec:Conclusions}
In this article, we have proved the contraction property for
various symmetric discontinuous Galerkin (DG) methods. Unlike in
the existing works for strongly penalized DG methods, we prove the
convergence of weakly penalized adaptive DG methods without
further assuming the stabilizing parameter is larger than what is
required for stability. Although the analysis in this article is
restricted to the lowest order case, we hope that similar ideas
may be used in higher order cases. We remark that
the convergence analysis of adaptive DG methods using strong jumps is still open
when the stabilizing parameter is chosen just according to the
stability of the method.

\noindent {\bf Acknowledgments}

We would like to thank Mark Ainsworth for many useful discussions.
The work was done while the first author visited Brown University
and Institute for Computational and Experimental Research in
Mathematics (ICERM) with the help of the Indo-US Virtual Institute
of Mathematical and Statistical Sciences. We would like to thank
Govind Menon and Govindan Rangarajan for facilitating this visit.

%

\end{document}